\documentclass[12pt]{amsart}

\usepackage{amsfonts}
\usepackage{amsmath}
\usepackage{amssymb}
\usepackage{mathrsfs}
\usepackage{amscd}
\usepackage[all]{xy}
\usepackage[pdftex,final]{graphicx}
\input cyracc.def

\newtheorem{lem}{Lemma}[section]
\newtheorem{thm}[lem]{Theorem}

\newtheorem{cor}[lem]{Corollary}
\theoremstyle{definition}

\newtheorem{rem}[lem]{Remark}

\title{The third Homology of the special linear group of a field}
\author{Kevin Hutchinson,  \   Liqun Tao}
\address{School of Mathematical Sciences,
 University College Dublin}
\email{kevin.hutchinson@ucd.ie, \ lqtao@ucd.ie}
\date{\today}


\newcommand{\imp}{\Longrightarrow}

\newcommand{\Q}{\Bbb{Q}}

\newcommand{\Z}{\Bbb{Z}}

\newcommand{\R}{\Bbb{R}}

\newcommand{\ee}{\epsilon}

\newcommand{\xgen}[2]{X_{#1}(#2)}
\newcommand{\Xgen}[1]{X_{#1}}

\newcommand{\cgen}[2]{C_{#1}(#2)}
\newcommand{\Cgen}[1]{C_{#1}}

\newcommand{\sgen}[1]{\mathrm{S}_{#1}}

\newcommand{\ddet}{\overline{\det}}

\newcommand{\set}[2]{\{#1 \mid #2 \}}

\renewcommand{\ker}[1]{\mathrm{Ker}(#1)}
\newcommand{\image}[1]{\mathrm{Im}(#1)}

\newcommand{\coker}[1]{\mathrm{Coker}(#1)}

\newcommand{\gr}[2]{#1[#2]}
\newcommand{\GR}[2]{#1\left[ #2\right]}
\newcommand{\aug}[1]{\mathcal{I}_{#1}}

\newcommand{\gw}[1]{\mathrm{GW}(#1)}
\newcommand{\wgen}[1]{\langle #1 \rangle}
\newcommand{\pfist}[1]{\langle\langle #1 \rangle\rangle}

\newcommand{\ffist}[1]{\langle\langle #1 \rangle\rangle}
\newcommand{\fgen}[1]{\langle #1 \rangle}

\newcommand{\filt}[2]{F_{#1}{#2}}

\newcommand{\ext}[2]{\bigwedge^{#1}\left({#2}\right)}

\newcommand{\proj}[2]{\mathbb{P}^{#1}(#2)}
\newcommand{\p}[1]{\tilde{#1}}

\newcommand{\kf}[2]{K_{#1} (#2)}

\newcommand{\mwk}[2]{K^{\mathrm{\small MW}}_{#1}({#2})}

\newcommand{\milk}[2]{K^{\mathrm{\small M}}_{#1}({#2})}
\newcommand{\milkt}[2]{k^{\mathrm{\small M}}_{#1}({#2})}
\newcommand{\indk}[2]{K_{#1}(#2)^{\mathrm{\small ind}}}
\newcommand{\qmilk}[2]{\bar{K}^{\mathrm{\small M}}_{#1}({#2})}

\newcommand{\ho}[3]{\mathrm{H}_{#1}(#2,#3 )}
\newcommand{\hoz}[2]{\mathrm{H}_{#1}(#2, \Z)}

\newcommand{\barhoz}[2]{\overline{\mathrm{H}}_{#1}(#2, \Z)}

\newcommand{\genl}[2]{\mathrm{GL}_{#1}(#2)}
\newcommand{\specl}[2]{\mathrm{SL}_{#1}(#2)}
\newcommand{\gnl}[1]{\mathrm{GL}(#1)}
\newcommand{\spcl}[1]{\mathrm{SL}(#1)}
\newcommand{\mat}[3]{\mathrm{M}(#1,#2,#3)}

\begin{document}

\maketitle

\begin{abstract}
We prove that for any infinite field $F$, the map $\ho{3}{\specl{n}{F}}{\Z}\to\ho{3}{\specl{n+1}{F}}{\Z}$ is an 
isomorphism for all $n\geq 3$. 
When $n=2$ the cokernel of this map is naturally isomorphic to $2\cdot \milk{3}{F}$, where $\milk{n}{F}$ is the $n$th Milnor $K$-group 
of $F$ .
We deduce that the natural homomorphism from $\hoz{3}{\specl{2}{F}}$
to the indecomposable $K_3$ of $F$, $\indk{3}{F}$, is surjective for any infinite field $F$.
\end{abstract}
\section{Introduction}

A well-known result of A. A. Suslin (\cite{sus:homgln}) relates the Milnor $K$-theory of a field, $F$,  
to the question of homology stability for the general linear group 
of the field. He shows that the maps $\hoz{p}{\genl{n}{F}}\to\hoz{p}{\genl{n+1}{F}}$ are isomorphisms for $n\geq p$ and that, when 
$n=p-1$, the cokernel is naturally isomorphic to $\milk{p}{F}$. In fact, if we let $H_p(F):= \hoz{p}{\genl{p}{F}}/\hoz{p}{\genl{p-1}{F}}$,
his arguments show that there is an isomorphism of graded rings $H_\bullet(F)\cong \milk{\bullet}{F}$ (where the multiplication on the 
first term comes from direct sum of matrices and cross product on homology). In particular, the non-negatively graded ring $H_\bullet(F)$
is generated in dimension $1$.

Recent work of Barge and Morel (\cite{barge:morel}) suggests that Milnor-Witt $K$-theory may play a somewhat analogous role for the homology
of the special linear group. The Milnor-Witt  $K$-theory of $F$ is a $\Z$-graded ring $\mwk{\bullet}{F}$ surjecting naturally onto 
Milnor $K$-theory. It arises as a ring of operations in stable motivic homotopy theory. (For a definition see section \ref{sec:background}
below, and for more details see \cite{morel:trieste,morel:puiss,morel:a1}.)
  Let $SH_p(F):=\hoz{p}{\specl{p}{F}}/\hoz{p}{\specl{p-1}{F}}$ for $p\geq 1$, and let $SH_0(F)=\GR{\Z}{F^\times}$
for convenience.
Barge and Morel construct a map of graded algebras $SH_\bullet(F)\to\mwk{\bullet}{F}$ for which the square
\begin{eqnarray*}
\xymatrix{
SH_\bullet(F)\ar[r]\ar[d]
&
\mwk{\bullet}{F}\ar[d]\\
H_\bullet(F)\ar[r]
&
\milk{\bullet}{F}
}
\end{eqnarray*}
commutes.

A result of Suslin (\cite{sus:tors}) implies that the map $\hoz{2}{\specl{2}{F}}=SH_2(F)\to\mwk{2}{F}$ is an isomorphism. Since positive-dimensional Milnor-Witt $K$-theory is generated in dimension $1$, it follows that 
the map of Barge and Morel is surjective in even dimensions  
greater than or equal to $2$. They ask the question whether it is in 
fact an isomorphism in even dimensions.

Of course, since $\specl{1}{F}=\{ 1\}$, $SH_1(F)=0$, and there is no possiblity that the map is an isomorphism in odd dimensions.  
Furthermore, Barge and Morel make the point that their map is also a map of graded 
$\GR{\Z}{F^\times}$-modules, but that $(F^\times)^2$ acts trivially on $\mwk{\bullet}{F}$
while $(F^\times)^n$ acts trivially on $SH_n(F)$.
It follows that, in odd dimensions, the 
image of their homomorphism must be contained in the $F^\times$-invariants, $\mwk{n}{F}^{F^\times}$, of Milnor-Witt $K$-theory. 
Now $2\milk{n}{F}$ occurs  naturally as a subgroup of  $\mwk{n}{F}^{F^\times}$, and these groups are very often equal (the obstruction 
to equality in general being a small subgroup of the torsion subgroup of the Witt ring of the field). The question thus arises whether the 
Barge-Morel homomorphism induces an isomorphism $SH_n(F)\cong 2\milk{n}{F}$ in odd dimensions greater than or equal to $3$.

This paper is intended primarily as a first step towards answering the question of Barge and Morel. The authors believe that 
it will be useful to answer the case $n=3$ in order to attack the higher-dimensional - even or odd - cases. Our main result is that the map 
$SH_3(F)\to \mwk{3}{F}$  induces an isomorphism $SH_3(F)\cong 2\milk{3}{F}$.  In the process, we also prove the homology 
stability result that the maps 
$\hoz{3}{\specl{n}{F}}\to\hoz{3}{\specl{n+1}{F}}$ are isomorphisms for all $n\geq 3$. 
Chih-Han Sah, in \cite{sah:discrete3}, states (without proving it) that these maps are isomorphisms for all $n\geq 4$. So our contribution here is to extend the known stability range by $1$. 

Observe that it easily follows from our main 
result that, for all odd $n$ greater than or equal to $3$, the image of the Barge-Morel map contains $2\milk{n}{F}$ (since 
in Milnor-Witt $K$-theory, $(2\milk{r}{F})\cdot\mwk{s}{F}=2\milk{r+s}{F}$).
As another application, we deduce  that for any infinite field $F$ the natural map 
$\ho{0}{F^\times}{\hoz{3}{\specl{2}{F}}}\to\indk{3}{F}$ (\cite{sus:bloch,sah:discrete3,pev:indk3,mirzaii:3rd}) is surjective. Suslin has asked the question whether this map is always an isomorphism.  

The homological techniques we use in this paper are standard in this area and go back to the earliest results on homology stability for 
linear groups. The new ingredient is our use of the Milnor conjecture and Milnor-Witt 
$K$-theory to prove  Theorem \ref{thm:main} below. The discussion above suggests that the appearance of Milnor-Witt $K$-theory in this context is not  merely incidental.

\section{Notation and Background Results}\label{sec:background}
\subsection{Group Rings and Grothendieck-Witt Rings}
For a group $G$, we let $\GR{\Z}{G}$ denote the corresponding integral groupring. It has an additive 
$\Z$-basis consisting of 
the elements $\fgen{g}$, $g\in G$, with the natural multiplication. (We use this notation in order to distinguish the elements 
$\fgen{1-a}$ from $1-\fgen{a}$ when $G$ is the multiplicative group, $F^\times$, of a field $F$.) 
There is an augmentation homomorphism $\epsilon:
\GR{\Z}{G}\to
\Z$, $\fgen{g}\mapsto 1$, whose kernel is the augmentation ideal $\aug{G}$, generated by the elements 
$\ffist{g}:=\fgen{g}-1$.    

The Grothendieck-Witt ring of a field $F$ is the Grothendieck group, $\gw{F}$, of the set of isometry classes of nondgenerate symmetric 
bilinear forms under orthogonal sum. Tensor product of forms induces a natural multiplication on the group. As an abstract ring, 
this can be described as the quotient of the ring $\GR{\Z}{F^\times/(F^\times)^2}$ by the ideal generated by the elements 
$\ffist{a}\cdot\ffist{1-a}$, $a\not= 0,1$. (This is just a mild reformulation of the presentation given in Lam, \cite{lam:intro}, Chapter II,
 Theorem 4.1.) Here, the induced ring homomorphism $\GR{\Z}{F^\times}\to \GR{\Z}{F^\times/(F^\times)^2}\to
\gw{F}$, sends $\fgen{a}$ to the class of the $1$-dimensional form with matrix $[a]$. This class is (also) denoted $\wgen{a}$. $\gw{F}$
is again an augmented ring and the augmentation ideal, $I(F)$, - also called the \emph{fundamental ideal} - 
is generated by \emph{Pfister $1$-forms}, $\pfist{a}$. It follows that the 
$n$-th power, $I^n(F)$, of this ideal is generated by \emph{Pfister} $n$-forms $\pfist{a_1,\ldots,a_n}:=\pfist{a_1}\cdots\pfist{a_n}$.

We will use below the following result on the powers of the fundamental ideal.

\begin{lem}[\cite{hutchinson:tao}, Cor 2.16] \label{lem:huttao}
 For all $n\geq 0$, $I^{n+2}(F)$
 is generated, as an additive group, by Pfister forms $\pfist{a_0,\ldots,a_{n+1}}$ subject to the following relations:
\begin{enumerate}
\item[(a)] $\pfist{a_0,\ldots,a_{n+1}}=0\mbox{ if } a_{n+1}\in (F^\times)^2$
\item[(b)] $\pfist{a_0,\ldots,a_i,a_{i+1},\ldots,a_{n+1}}=\pfist{a_0,\ldots,a_{i+1},a_i,\ldots,a_{n+1}}$
\item[(c)] $\pfist{a_0,\ldots,a_n,a_{n+1}a'_{n+1}}+\pfist{a_0,\ldots,a_{n-1},a_{n+1},a'_{n+1}}=\pfist{a_0,\ldots,a_na_{n+1},a'_{n+1}}+
\pfist{a_0,\ldots,a_n,a_{n+1}}$
\item[(d)] $\pfist{a_0,\ldots,a_n,a_{n+1}}=\pfist{a_0,\ldots,a_n,(1-a_n)a_{n+1}}$
\end{enumerate}

\end{lem}

Now let $\textbf{h}:=\wgen{1}+\wgen{-1}=\pfist{-1}+2\in \gw{F}$. Then $\textbf{h}\cdot I(F)=0$, and the \emph{Witt ring} of $F$ is the ring
\[
W(F):=\frac{\gw{F}}{\langle \textbf{h}\rangle}=\frac{\gw{F}}{\textbf{h}\cdot \Z}.
\] 
Since $\textbf{h}\mapsto 2$ under the augmentation, there is a natural ring homomorphism $W(F)\to \Z/2$. The fundamental ideal $I(F)$ of 
$\gw{F}$ 
maps isomorphically to the kernel of this ring homomorphism under the map $\gw{F}\to W(F)$, and we also let $I(F)$ denote this ideal.

For $n\leq 0$, we define $I^n(F):=W(F)$. The graded additive group $I^\bullet(F)=\{ I^n(F)\}_{n\in\Z}$ is given the structure of a 
commutative graded ring using the natural graded multiplication induced from the multiplication on $W(F)$. In particular, if we let 
$\eta\in I^{-1}(F)$ be the element corresponding to $1\in W(F)$, then multiplication by $\eta:I^{n+1}(F)\to I^n(F)$ is just the natural 
inclusion.

\subsection{Milnor $K$-theory and Milnor-Witt $K$-theory}

The Milnor ring of a field $F$  (see \cite{milnor:intro}) is the graded ring
\(
\milk{\bullet}{F}
\)
with the following presentation:

Generators: $\{ a\}$ , $a\in F^\times$, in dimension $1$.

Relations:
\begin{enumerate}
\item[(a)] $\{ ab\} =\{ a\} +\{ b\}$ for all $a,b\in F^\times$.
\item[(b)] $\{a\}\cdot\{1-a\}=0$ for all $a\in F^\times\setminus\{ 1\}$.
\end{enumerate}
The product $\{ a_1\}\cdots\{ a_n\}$ in $\milk{n}{F}$ is also written $\{ a_1,\ldots,a_n\}$. So $\milk{0}{F}=\Z$ and $\milk{1}{F}$ is 
an additive group isomorphic to $F^\times$.

We let $\milkt{\bullet}{F}$ denote the graded ring $\milk{\bullet}{F}/2$ and let $i^n(F):=I^n(F)/I^{n+1}(F)$, so that $i^\bullet(F)$ is a 
non-negatively graded ring. 

In the 1990s, Voevodsky and his collaborators proved a fundamental and deep theorem - originally conjectured by Milnor (\cite{milnor:quad}) - 
relating Milnor $K$-theory to quadratic form theory:

\begin{thm}[\cite{voevodsky:orlovvishik}]
There is a natural isomorphism of graded rings $\milkt{\bullet}{F}\cong i^\bullet(F)$ sending $\{ a\}$ to $\pfist{a}$.

In particular for all $n\geq 1$ we have a natural identification of $\milkt{n}{F}$ and $i^n(F)$ under which the symbol $\{ a_1,\ldots,a_n\}$ 
corresponds to the class of the form $\pfist{a_1,\ldots,a_n}$.
\end{thm}

The Milnor-Witt $K$-theory of a field is the graded ring $\mwk{\bullet}{F}$ with the following presentation (due to F. Morel and M. Hopkins,
see \cite{morel:trieste}):

Generators: $[a]$, $a\in F^\times$, in dimension $1$ and a further generator $\eta$ in dimension $-1$.

Relations: 
\begin{enumerate}
\item[(a)] $[ab]=[a]+[b]+\eta\cdot [a]\cdot [b]$ for all $a,b\in F^\times$
\item[(b)] $[a]\cdot[1-a]=0$ for all $a\in F^\times\setminus\{ 1\}$
\item[(c)] $\eta\cdot [a]=[a]\cdot \eta$ for all $a\in F^\times$
\item[(d)] $\eta\cdot h=0$, where $h=\eta\cdot [-1] +2\in \mwk{0}{F}$.
\end{enumerate}

Clearly there is a unique surjective homomorphism of graded rings $\mwk{\bullet}{F}\to\milk{\bullet}{F}$ sending $[a]$ to $\{ a\}$
 and inducing an isomorphism
\[
\frac{\mwk{\bullet}{F}}{\langle \eta \rangle}\cong \milk{\bullet}{F}.
\]

Furthermore, there is a natural surjective homomorphism of graded rings $\mwk{\bullet}{F}\to I^\bullet(F)$ sending $[a]$ to $\pfist{a}$ and 
$\eta$ to $\eta$. Morel shows that there is an induced isomorphism of graded rings
\[
\frac{\mwk{\bullet}{F}}{\langle h \rangle}\cong I^{\bullet}(F).
\]


The main structure theorem on Milnor-Witt $K$-theory is the following theorem of Morel:

\begin{thm}[Morel, \cite{morel:puiss}]
The commutative square of graded rings
\begin{eqnarray*}
\xymatrix{
\mwk{\bullet}{F}\ar[r]\ar[d]
&
\milk{\bullet}{F}\ar[d]\\
I^\bullet(F)\ar[r]
&
\milkt{\bullet}{F}
}
\end{eqnarray*}
is cartesian.
\end{thm}

Thus for each $n\in \Z$ we have an isomorphism
\[
\mwk{n}{F}\cong \milk{n}{F}\times_{i^n(F)}I^n(F).
\]

It follows, for example, that for all $n$, the kernel of $\mwk{n}{F}\to\milk{n}{F}$ is isomorphic to $I^{n+1}(F)$ and, 
for $n\geq 0$,  the inclusion
$I^{n+1}(F)\to\mwk{n}{F}$ is given by $\pfist{a_1,\ldots,a_{n+1}}\mapsto \eta[a_1]\cdots[a_n]$.

When $n=0$ we have an isomorphism of rings
\[
\gw{F}\cong W(F)\times_{\Z/2}\Z\cong \mwk{0}{F}.
\]
Under this isomorphism $\pfist{a}$ corresponds to $\eta [a]$ and $\wgen{a}$ corresponds to $\eta[a]+1$. 

Thus each $\mwk{n}{F}$ has the structure of a $\gw{F}$-module (and hence also of a $\GR{\Z}{F^\times}$-module), with the action given by 
$\ffist{a}\cdot([a_1]\cdots [a_n])=\eta[a] [a_1]\cdots [a_n]$.

Thus, for all $n\geq 0$,
\[
\aug{F^\times}\cdot\mwk{n}{F}=I(F)\cdot\mwk{n}{F}=\eta\cdot\mwk{n+1}{F}=I^{n+1}(F) 
\]
and hence 
\[
\ho{0}{F^\times}{\mwk{n}{F}}=\frac{\mwk{n}{F}}{\aug{F^\times}\cdot\mwk{n}{F}}\cong \milk{n}{F}.
\]

\subsection{Homology of Groups}

Given a group $G$ and a $\GR{\Z}{G}$-module $M$, $\ho{n}{G}{M}$ will denote the $n$th homology group of $G$ with coefficients in the 
module $M$. $B_\bullet(G)$ will denote the \emph{right bar resolution of $G$}: $B_n(G)$ is the free right $\GR{\Z}{G}$-module with basis
the elements $[g_1|\cdots |g_n]$, $g_i\in G$. ($B_0(G)$ is isomorphic to $\GR{\Z}{G}$ with generator the symbol $[\ ]$.) The boundary 
$d=d_n:B_n(G)\to B_{n-1}(G)$, $n\geq 1$, is given by 
\[
d([g_1|\cdots|g_n])= \sum_{i=o}^{n-1}(-1)^i[g_1|\cdots|\hat{g_i}|\cdots|g_n]+(-1)^n[g_1|\cdots|g_{n-1}]\fgen{g_n}.
\]  
The augmentation $B_0(G)\to \Z$ makes $B_\bullet(G)$ into a free resolution of the trivial $\GR{\Z}{G}$-module $\Z$, and thus 
$\ho{n}{G}{M}=H_n(B_\bullet(G)\otimes_{\GR{\Z}{G}}M)$.

If $C_\bullet = (C_q,d)$ is a non-negative complex of $\GR{\Z}{G}$-modules, then 
$E_{\bullet,\bullet}:=B_\bullet(G)\otimes_{\gr{\Z}{G}}C_\bullet$ 
is a double complex  of abelian groups.
Each of the two filtrations on $E_{\bullet,\bullet}$ gives a spectral sequence  converging to the homology of the total complex of 
$E_{\bullet,\bullet}$,  which is by definition, 
$\ho{\bullet}{G}{C}$. (see, for example, Brown, \cite{brown:coh}, Chapter VII).

The first spectral sequence has the form
\[
E^2_{p,q}=\ho{p}{G}{H_q(C)}\Longrightarrow \ho{p+q}{G}{C}.
\] 
In the special case that there is a weak equivalence $C_\bullet \to \Z$ (the complex consisting of the trivial module 
$\Z$ concentrated in dimension $0$), it follows that $\ho{\bullet}{G}{C}=\ho{\bullet}{G}{\Z}$. 

The second spectral sequence has the form 
\[
E^1_{p,q}=\ho{p}{G}{C_q}\Longrightarrow\ho{p+q}{G}{C}.
\]
Thus, if $C_\bullet$ is weakly equivalent to $\Z$, this gives a spectral sequence converging to $\ho{\bullet}{G}{\Z}$. 

In the next theorem, we summarize some of the main results of \cite{sus:homgln}:

\begin{thm}\label{thm:suslin} 
Let $T_n$ be the group of diagonal matrices in $\genl{n}{F}$.
\begin{enumerate}
\item[(a)] The natural maps $\hoz{p}{\genl{n}{F}}\to\hoz{p}{\genl{n+1}{F}}$ are isomorphisms if $n\geq p$.
\item[(b)] For any fixed $p$ and any $n\geq p$ there is a natural map $\ee=\ee_{p,n}: \hoz{p}{\genl{n}{F}}\to \milk{p}{F}$ whose kernel is the image of 
$\hoz{p}{\genl{p-1}{F}}$. 
\item[(c)] For $n\geq p$, the composite
\begin{eqnarray*}
\xymatrix{
\ext{p}{T_n}\ar[r]& \hoz{p}{T_n}\ar[r]&\hoz{p}{\genl{n}{F}}\ar[r]^-{\ee}&\milk{p}{F}\\
}
\end{eqnarray*}
is given by
\begin{eqnarray*}\label{form:mult} 
(a_{1,1},\ldots,a_{1,n})\wedge \cdots \wedge (a_{p,1},\ldots,a_{p,n})\mapsto
 \sum_{J=(j_1,\ldots,j_p)}
\left\{ a_{1,j_1},\ldots,a_{p,j_p}\right\}
\end{eqnarray*}
where $J$ ranges over all ordered $p$-tuples of distinct points of $\{ 1,\ldots,n\}$.
\item[(d)] For any $n\geq 3$ the composite
\begin{eqnarray*}
\xymatrix{
\hoz{2}{\specl{n}{F}}\ar[r]&\hoz{2}{\genl{n}{F}}\ar[r]^-{\ee}&\milk{2}{F}\\
}
\end{eqnarray*}
is an isomorphism.
\end{enumerate}
\end{thm}
\begin{proof}\ 

\begin{enumerate}

\item[(a)] \cite{sus:homgln}, Theorem 3.4 (c)
\item[(b)] \cite{sus:homgln}, Theorem 3.4 (d)
\item[(c)] 
Let $\rho_k=\rho_{k,n}:F^\times\to T_n$ denote the inclusion in the $k$th factor. 
Now, writing the operation in $\bigwedge^p T_n$ as addition, we have 
\[
(a_{1,1},\ldots,a_{1,n})\wedge \cdots \wedge (a_{p,1},\ldots,a_{p,n})=\sum_{1\leq k_1,\ldots,k_p\leq n}
\rho_{k_1}(a_{1,k_1})\wedge\cdots\wedge\rho_{k_p}(a_{p,k_p}).
\]

It follows from Corollary 2.7.2 of \cite{sus:homgln} that 
the image of $\rho_1(a_1)\wedge \cdots \wedge \rho_p(a_p)$ is $\{ a_1,\ldots,a_p\}$. Now if $J=(j_1,\ldots,j_p)$ is a $p$-tuple of distinct points of 
$\{ 1,\ldots,n\}$ there exists a permutation sending $j_i$ to $i$. The corresponding permutation on the factors of $T_n$ can be realised 
as conjugation by a permutation matrix in $\genl{n}{F}$. Since the induced automorphism on the integral homology of $\genl{n}{F}$ 
is trivial, it follows that the image of
$\rho_{j_1}(a_1)\wedge \cdots \wedge \rho_{j_p}(a_p)$ is again $\{ a_1,\ldots,a_p\}$.

On the other hand, if $k_1,\ldots,k_p$ are a set of natural numbers smaller than $n$ some two of which are equal, then the image of 
  $\rho_{k_1}(a_1)\wedge \cdots \wedge \rho_{k_p}(a_p)$ lies in the image of $\hoz{p}{\genl{p-1}{F}}$ and so is $0$.

\item[(d)] The Hochschild-Serre spectral sequence of the extension\\
 $1\to \specl{n}{F}\to\genl{n}{F}\to F^\times\to 1$ (see the next section) gives,
for $n\geq 2$, a split exact sequence
\[
0\to \ho{0}{F^\times}{\hoz{2}{\specl{n}{F}}}\to\hoz{2}{\genl{n}{F}}\to \hoz{2}{F^\times}\to 0.
\]
Since the map $\hoz{2}{\genl{1}{F}}\to\hoz{2}{F^\times}$ induced by the determinant is trivially an isomorphism, it follows from the results 
just cited
 that $\epsilon$ induces an isomorphism $\ho{0}{F^\times}{\hoz{2}{\specl{n}{F}}}\cong \milk{2}{F}$ for all $n\geq 2$. 

It can be shown that, for $n\geq 3$, $F^\times$ acts trivially on $\hoz{2}{\specl{n}{F}}$ (see for example Hutchinson, 
\cite{hutchinson:mat}, p199) and thus 
$\ho{0}{F^\times}{\hoz{2}{\specl{n}{F}}}=\hoz{2}{\specl{n}{F}}$ in this case.
\end{enumerate}
\end{proof}
\begin{rem}
On the other hand, the action of $F^\times$ on $\hoz{2}{\specl{2}{F}}$ is not generally trivial. In fact a result of 
Suslin (\cite{sus:tors}, Appendix ) 
essentially states that there is a natural isomorphism $\hoz{2}{\specl{2}{F}}\cong \mwk{2}{F}$ for which the 
diagram
\begin{eqnarray*}
\xymatrix{
\hoz{2}{\specl{2}{F}}\ar[r]^-{\cong}\ar[d]
&
\mwk{2}{F}\ar[d]\\
\hoz{2}{\specl{3}{F}}\ar[r]^-{\cong}
&
\milk{2}{F}
}
\end{eqnarray*}
commutes (for more details see also Mazzoleni, \cite{mazz:sus} or Hutchinson and Tao, \cite{hutchinson:tao}). 
It  follows that the kernel of $\hoz{2}{\specl{2}{F}}\to
\hoz{2}{\specl{3}{F}}$ is isomorphic to $I^3(F)$ and this group measures the nontriviality of the action of $F^\times$ on the first group. 
\end{rem}

\begin{rem}
For a different (but related) interpretation of Milnor $K$-theory in terms of the homology of orthogonal groups see Cathelineau, 
\cite{cath:milnor}.
\end{rem}

We also require the following lemma, essentially due to Suslin, in section \ref{sec:stab} below:

If $G$ is a subgroup of $\genl{r}{F}$, let ${G}^0$ denote $G\cap\specl{r}{F}$.

\begin{lem}\label{lem:tb}
Let $G_1\subset \genl{n}{F}$ and $G_2\subset \genl{m}{F}$ be subgroups both of which contain the group, $F^\times$, of scalar matrices. 
Let $M\subset\mat{n}{m}{F}$ be a vector subspace with the property that $G_1M=M=MG_2$. Then the  embedding of groups 
\[
{\begin{pmatrix}
G_1&0\\
0&G_2
\end{pmatrix}}^0
\to
\begin{pmatrix}
G_1&M\\
0&G_2
\end{pmatrix}^0
\] 
induces an isomorphism on integral homology.
\end{lem}

\begin{proof} This is Lemma 9 of \cite{hutchinson:mat}, which is a modified version of Lemma 1.8 of \cite{sus:tors}.
\end{proof}
\section{The Hochschild-Serre spectral sequences}
 For all $n\geq 2$, the short exact sequences 
 \begin{eqnarray*}
 \xymatrix{1\ar[r]&\specl{n}{F}\ar[r]&\genl{n}{F}\ar[r]^-{\det}&F^{\times}\ar[r]& 1}
 \end{eqnarray*}
 give Hochschild-Serre spectral sequences of the form
 \[
E^2_{p,q}(n)=\ho{p}{F^\times}{\hoz{q}{\specl{n}{F}}}\Rightarrow \hoz{p+q}{\genl{n}{F}}
 \]
 with differentials 
 \[
 d^r:E^r_{p,q}\to E^r_{p-r,q+r-1}.
 \]
The edge maps 
\[
D_p:\hoz{p}{\genl{n}{F}}\to E^\infty_{p,0} \subset E^2_{p,0}=\hoz{p}{F^\times}
\]
 are the maps induced by $\det: \genl{n}{F}\to F^\times$, and thus are surjective since $\det$ has a splitting. Thus 
 \[
 E^2_{p,0}=E^3_{p,0}=\cdots = E^\infty_{p,0} \mbox{ for all } p
 \]
 and the differentials $d^r$ leaving $E^r_{p,0}$ are zero for all $r\geq 2$.
 
 Of course, the terms $E^r_{p,1}$ are  zero for all $r\geq 2,p\geq 0$ since $\hoz{1}{\specl{n}{F}}=0$.

We will let  
\[
\barhoz{q}{\specl{n}{F}}:= E^\infty_{0,q} = \image{\hoz{q}{\specl{n}{F}}\to \hoz{q}{\genl{n}{F}}}.
\]

 Thus the $E^\infty$-term of the $n$th spectral sequence has the form
 
 \begin{eqnarray*}
{ 
 \xymatrix{
 \barhoz{4}{\specl{n}{F}}
 &\cdots &\cdots &\cdots &\cdots\\
 \barhoz{3}{\specl{n}{F}}
 &E^\infty_{1,3}(n)&\cdots &\cdots&\cdots \\
 \barhoz{2}{\specl{n}{F}}
 &
 \ho{1}{F^\times}{\hoz{2}{\specl{n}{F}}}
 &E^\infty_{2,2}(n) &\cdots&\cdots \\
 0&0&0&0&0\\
 \Z
 &{F^\times}
 &\hoz{2}{F^\times}
 &\hoz{3}{F^\times}
 &\hoz{4}{F^\times}\\
 }
} 
\end{eqnarray*}

The spectral sequence thus determines a filtration on $\hoz{p}{\genl{n}{F}}$, with 
\begin{eqnarray*}
 \filt{p}{\hoz{p}{\genl{n}{F}}} & = & \hoz{p}{\genl{n}{F}}  \\ 
\filt{p-1}{\hoz{p}{\genl{n}{F}}}&=&\filt{p-2}{\hoz{p}{\genl{n}{F}}}=\ker{D_p}\\
  \filt{0}{\hoz{p}{\genl{n}{F}}}&=&\barhoz{p}{\specl{n}{F}}=E^\infty_{0,p}.\\
\end{eqnarray*}

By Theorem \ref{thm:suslin} (1), the natural maps $\hoz{p}{\genl{n}{F}}\to\hoz{p}{\genl{n+1}{F}}$ are isomorphisms for $n\geq p$. 
Since these maps are compatible with the corresponding maps of Hochschild-Serre spectral sequences, 
it follows that the induced maps $E^\infty_{p,q}(n)\to E^\infty_{p,q}(n+1)$ are isomorphisms if $p+q\leq n$. 

In particular, the maps $\barhoz{3}{\specl{n}{F}}\to\barhoz{3}{\specl{n+1}{F}}$ are isomorphisms for all $n\geq 3$.

Furthermore, it follows from the remarks above that
 the maps $E^r_{p,q}(n)\to E^r_{p,q}(n+1)$ are isomorphisms for all $r\geq 2$ and $n\geq 1$ if 
$q\in\{ 0,1\}$.

Observe that the only possible nonzero differential arriving at the $(0,3)$ position is 
\[
d^2=d^2_{2,2}(n):E^2_{2,2}=\ho{2}{F^\times}{\hoz{2}{\specl{n}{F}}}\to E^2_{0,3}=\ho{0}{F^\times}{\hoz{3}{\specl{n}{F}}}
\]
and thus for all $n$
 \[
 \barhoz{3}{\specl{n}{F}}=
 \frac{\ho{0}{F^\times}{\hoz{3}{\specl{n}{F}}}}{\image{d^2}}.
 \]

 \begin{lem}\label{lem:d2}
 For all $n\geq 3$, $E^\infty_{0,3}(n)=\barhoz{3}{\specl{n}{F}}=
 \ho{0}{F^\times}{\hoz{3}{\specl{n}{F}}}$.
 \end{lem}
 \begin{proof}
By the preceding remarks, this amounts to showing that $d^2_{2,2}(n)=0$ for all $n\geq 3$.

 The natural map $\hoz{p}{\spcl{F}}\to\hoz{p}{\gnl{F}}$ is split injective. This follows from letting $n$ go to infinity in the commutative triangle
 \begin{eqnarray*}
 \xymatrix{\hoz{p}{\specl{n}{F}}\ar[rr]\ar[rd]&&\hoz{p}{\genl{n}{F}}\ar[ld]^{\alpha_p}\\
 &\hoz{p}{\specl{n+1}{F}}&\\
 }
 \end{eqnarray*}
 where $\alpha_p$ is the map induced on homology by the embedding
 \[
 \alpha:\genl{n}{F}\to\specl{n+1}{F},\quad 
 A\mapsto 
 \begin{pmatrix}
 A&0\\
 0&\det{A}^{-1}\\
 \end{pmatrix}.
 \]
 It follows that the differential 
 \[
 d^2:E^2_{2,2}=\ho{2}{F^\times}{\hoz{2}{\spcl{F}}}\to E^2_{0,3}=\ho{0}{F^\times}{\hoz{3}{\spcl{F}}}.
 \]
 in the spectral sequence of the extension
\(
1\to \spcl{F}\to\gnl{F}\to F^\times\to 1
\) 
is zero, since $E^\infty_{0,3}=E^2_{0,3}=\hoz{3}{\spcl{F}}$. 
Furthermore, it follows that $E^\infty_{2,2}= \ker{d^2}=E^2_{2,2}$ since no other nonzero differentials arrive at or leave this position.

By the remarks above, the natural maps $E^\infty_{p,q}(n)\to E^\infty_{p,q}$ are isomorphisms if either $n\geq p+q$ or $q\leq 1$.

In particular we have, for $n\geq 4$, 
\[
\ho{2}{F^\times}{\milk{2}{F}}=E^2_{2,2}(n)
\supset E^\infty_{2,2}(n)\cong E^\infty_{2,2}=\ho{2}{F^\times}{\milk{2}{F}}
\]
(since $\hoz{2}{\specl{n}{F}}\cong\hoz{2}{\spcl{F}}\cong \milk{2}{F}$ by \ref{thm:suslin} (4)) 
so that $E^2_{2,2}(n)=E^\infty_{2,2}(n)$ and hence $d^2_{2,2}(n)=0$ in the case 
$n\geq 4$. 

For $n=3$ the cokernel of the map $\hoz{4}{\genl{3}{F}}\to \hoz{4}{\genl{4}{F}}$ is 
$\milk{4}{F}$ by \ref{thm:suslin} (2). By the remarks above, 
the cokernel of the map $\filt{2}{\hoz{4}{\genl{3}{F}}}\to \filt{2}{\hoz{4}{\genl{4}{F}}}$ is also $\milk{4}{F}$.


However, since the composite $\hoz{2}{\specl{2}{F}}\to\hoz{2}{\genl{2}{F}}\to \milk{2}{F}$ is surjective, it follows, taking products, that 
the map 
\[
\hoz{2}{\specl{2}{F}}\times\hoz{2}{\specl{2}{F}}\to\hoz{4}{\specl{4}{F}}\to\hoz{4}{\genl{4}{F}}\to\milk{4}{F}
\]
is composite, and hence that the cokernel of the map  $\filt{0}{\hoz{4}{\genl{3}{F}}}\to 
\filt{0}{\hoz{4}{\genl{4}{F}}}$ maps onto $\milk{4}{F}$. Thus the natural map 
$E^\infty_{2,2}(3)\to E^\infty_{2,2}(4)=\ho{2}{F^\times}{\milk{2}{F}}$ is surjective. 
As in the case $n\geq 4$, this forces $d^2_{2,2}(3)=0$.

 \end{proof}
 
For each $n$ the spectral sequence gives  a natural short exact sequence
\[
0\to \barhoz{3}{\specl{n}{F}}\to \filt{1}{\hoz{3}{\genl{n}{F}}}\to \ho{1}{F^\times}{\hoz{2}{\specl{n}{F}}}\to 0.
\]

Now the map of extensions
\begin{eqnarray*}
 \xymatrix{
 1\ar[r]&\specl{2}{F}\ar[d]\ar[r]&\genl{2}{F}\ar[d]\ar[r]^-{\det}&F^{\times}\ar[r]\ar[d]^{=}& 1\\
1\ar[r]&\specl{3}{F}\ar[r]&\genl{3}{F}\ar[r]^-\det& F^{\times}\ar[r]& 1}
 \end{eqnarray*}
induces a map of corresponding Hochschild-Serre spectral sequences and hence a commutative diagram with exact rows and column:
\begin{eqnarray}\label{diag:23}
\xymatrix{
&0\ar[r]&\barhoz{3}{\specl{2}{F}}\ar[r]\ar[d]&\filt{1}{\hoz{3}{\genl{2}{F}}}\ar[r]\ar[d]&\ho{1}{F^\times}{\hoz{2}{\specl{2}{F}}}
\ar[r]\ar[d]&0\\
&0\ar[r]&\barhoz{3}{\specl{3}{F}}\ar[r]&\filt{1}{\hoz{3}{\genl{3}{F}}}\ar[r]\ar[d]&\ho{1}{F^\times}{\hoz{2}{\specl{3}{F}}}\ar[r]&0\\
&&&\milk{3}{F}\ar[d]&&\\
&&&0&&\\
}
\end{eqnarray}

Note that $\ho{1}{F^\times}{\hoz{2}{\specl{3}{F}}}=\ho{1}{F^\times}{\milk{2}{F}}=F^\times\otimes\milk{2}{F}$, and hence there is a natural 
multiplication map $\mu:\ho{1}{F^\times}{\hoz{2}{\specl{3}{F}}}\to \milk{3}{F}$.

\begin{thm}\label{thm:main}
The natural inclusion $\specl{2}{F}\to\specl{3}{F}$ induces a short exact sequence 
\begin{eqnarray*}
\xymatrix{
0\ar[r]&\ho{1}{F^\times}{\hoz{2}{\specl{2}{F}}}\ar[r]&\ho{1}{F^\times}{\hoz{2}{\specl{3}{F}}}
\ar[r]^-{\mu}&\milkt{3}{F}\ar[r]&0\\
}
\end{eqnarray*}
\end{thm}
\begin{proof} 
Recall that there is a natural isomorphism of $\gr{\Z}{F^\times}$-modules 
 $\hoz{2}{\specl{2}{F}}\cong \mwk{2}{F}$ such that the map $\hoz{2}{\specl{2}{F}}\to\hoz{2}{\specl{3}{F}}$ corresponds to the 
natural surjection $\mwk{2}{F}\to\milk{2}{F}$.
 
 Let $I=I(F)$ be the augmentation ideal of the Grothendieck-Witt ring, $\gw{F}$, 
 of   $F$. For any $n\geq 0$, the exact sequence of $\gr{\Z}{F^\times}$-modules
 \[
 0\to I^{n+1}\to \mwk{n}{F}\to\milk{n}{F}\to 0
 \]
 gives an exact homology sequence
 \begin{eqnarray*}
 \ho{1}{F^\times}{\mwk{n}{F}}\to \ho{1}{F^\times}{\milk{n}{F}}\to \ho{0}{F^\times}{I^{n+1}}\to \ho{0}{F^\times}{\mwk{n}{F}}
\to\milk{n}{F}\to 0.
 \end{eqnarray*}

Now,
\begin{eqnarray*}
\ho{0}{F^\times}{\mwk{n}{F}}=\frac{\mwk{n}{F}}{\aug{F^\times}\cdot \mwk{n}{F}}
= \frac{\mwk{n}{F}}{I^{n+1}}
= \milk{n}{F}
\end{eqnarray*}
while
\begin{eqnarray*}
\ho{0}{F^\times}{I^{n+1}}=\frac{I^{n+1}}{\aug{F^\times}\cdot I^{n+1}}
= \frac{I^{n+1}}{I^{n+2}}\cong \milkt{n+1}{F}.
\end{eqnarray*}

Thus for all $n\geq 0$, there is an exact sequence
 \begin{eqnarray*}
 \ho{1}{F^\times}{\mwk{n}{F}}\to \ho{1}{F^\times}{\milk{n}{F}}\to \milkt{n+1}{F}\to 0.
 \end{eqnarray*}
 
 Let $A(n)$ be the statement that the leftmost arrow is an injection. Thus we require, in particular, to prove the statement $A(2)$. 

For $n\geq 0$, let $B(n)$ be the statement that the inclusion $I^{n+1}\to I^n$ induces the zero map $\ho{1}{F^\times}{I^{n+1}}\to \ho{1}{F^\times}{I^n}$.

We first show that $B(n+1)\imp A(n)$ for all $n\geq 1$:

 Suppose that $B(n+1)$ holds. Then the short exact sequence 
 \[
 0\to I^{n+2}\to I^{n+1}\to \milkt{n+1}{F}\to 0
 \]
 gives the exact homology sequence 
 \begin{eqnarray*}
 0\to \ho{1}{F^\times}{I^{n+1}}\to\ho{1}{F^\times}{\milkt{n+1}{F}}\to \milkt{n+2}{F}\to 0
 \end{eqnarray*}
(using the fact that $\ho{0}{F^\times}{I^{n}}= I^n/I^{n+1}=\milkt{n}{F}$.) 

Thus we have 
\begin{eqnarray*}
\ho{1}{F^\times}{I^{n+1}}& \cong & \ker{\ho{1}{F^\times}{\milkt{n+1}{F}}\to \milkt{n+2}{F}}\\
      & = & \ker{\mu_{n+2}:F^\times\otimes\milkt{n+1}{F}\to \milkt{n+2}{F}}
\end{eqnarray*}
 where $\mu  = \mu_{n+2}$ denotes the natural multiplication.
 
 On the other hand, from the long exact homology sequence associated to \\
 $0\to I^{n+1}\to \mwk{n}{F}\to \milk{n}{F}\to 0$, $A(n)$ is equivalent to the surjectivity of the connecting homomorphism
 \(
 \delta_n:\ho{2}{F^\times}{\milk{n}{F}}\to\ho{1}{F^\times}{I^{n+1}}.
 \)
 
 Let $\rho_n$ denote the composite $\ho{2}{F^\times}{\milk{n}{F}}\to\ho{1}{F^\times}{I^{n+1}}\to 
 \ho{1}{F^\times}{\milkt{n+1}{F}}=F^\times\otimes\milkt{n+1}{F}$.
 Then $A(n)$ holds if and only if the image of $\rho_n$ is the kernel of $\mu_{n+1}$. 
 
 Since $\milk{n}{F}$ is a trivial $\gr{\Z}{F^\times}$-module, there is a natural homomorphism
 \[
 \ext{2}{F^\times}\otimes\milk{n}{F}=\hoz{2}{F^\times}\otimes \milk{n}{F}\to \ho{2}{F^\times}{\milk{n}{F}}
 \]
 described at the level of the bar resolution by 
 \[
 (a\wedge b)\otimes z \mapsto \left([a|b]-[b|a]\right)\otimes z.
 \]
 
 We calculate the image of $(a\wedge b)\otimes \{ x_1,\ldots, x_n\}$ under $\delta_n$: As noted this is represented by the element 
 \[
 \left([a|b]-[b|a]\right)\otimes\{ x_1,\ldots, x_n\}\in B_2(F^\times)\otimes_{\gr{\Z}{F^\times}}\milk{n}{F}
 \]
 which lifts to 
  \[
 \left([a|b]-[b|a]\right)\otimes [x_1]\cdots [x_n]\in B_2(F^\times)\otimes_{\gr{\Z}{F^\times}}\mwk{n}{F}.
 \]
 This maps, by the boundary homomorphism, to the following element in $B_1(F^\times)\otimes\milk{n}{F}$:
 \begin{eqnarray*}
& \left([b]-[ab]+[a]\fgen{b}-[a]+[ba]-[b]\fgen{a}\right)\otimes [x_1]\cdots [x_n]\\
 = &\left([a]\ffist{b}-[b]\ffist{a}\right)\otimes [x_1]\cdots [x_n]\\
 =& [a]\otimes \ffist{b}[x_1]\cdots [x_n] - [b]\otimes \ffist{a}[x_1]\cdots [x_n].
 \end{eqnarray*}
Recall that the inclusion $I^{n+1}\to\mwk{n}{F}$ is given by 
\[
\pfist{a_0,\ldots,a_n}\mapsto \pfist{a_0}[a_1]\cdots [a_n] = \eta [a_0][a_1]\cdots [a_n].
\]

Thus 
\[
\delta_n\left( (a\wedge b)\otimes \{ x_1,\ldots, x_n\}\right) = [a]\otimes\pfist{b,x_1,\ldots,x_n}-[b]\otimes\pfist{a,x_1,\ldots,x_n}\in
\ho{1}{F^\times}{I^{n+1}}.
\] 
hence

\begin{eqnarray*}
\rho_n\left( (a\wedge b)\otimes \{ x_1,\ldots, x_n\}\right)& = & a\otimes\{ b,x_1,\ldots,x_n\}-b\otimes\{ a,x_1,\ldots,x_n\}\\
& = & a\otimes\{ b,x_1,\ldots,x_n\}+b\otimes\{ a,x_1,\ldots,x_n\}\in
F^\times\otimes\milkt{n+1}{F}.
\end{eqnarray*}

Since $\milk{n+2}{F}$ is naturally isomorphic to the $(n+2)$-nd tensor power of $F^\times$ modulo the two families of relations 
\begin{eqnarray*}
a_1\otimes \cdots a_{n+1}\otimes (1-a_{n+1})=0,\qquad a_i\in F^\times, a_{n+1}\not= 1\\
a_1\otimes \cdots a_i\otimes a_{i+1}\otimes \cdots \otimes a_{n+2}+a_1\otimes \cdots a_{i+1}\otimes a_i\otimes \cdots \otimes a_{n+2}, 
\qquad a_i\in F^\times  
\end{eqnarray*}
it easily follows that $\ker{F^\times\otimes \milk{n+1}{F}\to\milk{n+2}{F}}$ is generated by the elements 
\[
a\otimes\{ b,x_1,\ldots,x_n\}+b\otimes\{ a,x_1,\ldots,x_n\}.
\]

This shows that the image of $\rho_n$ is equal to the kernel of $\mu_{n+2}$, and thus that $A(n)$ holds.

Thus, to complete the proof of the theorem, we show that $B(n)$ holds for all $n\geq 1$ (and, in particular, that $B(3)$, 
and hence $A(2)$ holds).

For any $\gr{\Z}{G}$-module 
\[
\ho{1}{G}{M}=\ker{\aug{G}\otimes_{\gr{\Z}{G}}M\to \aug{G}\cdot M}.
\]
Thus
\[
\ho{1}{F^\times}{I^{n+1}}=\ker{\aug{F^\times}\otimes_{\gr{\Z}{F^\times}}I^{n+1}\to I^{n+2}}.
\]

From Lemma \ref{lem:huttao} it easily follows that there is a well-defined homomorphism (for $n\geq 1$)
\begin{eqnarray*}
\phi_n: I^{n+2}&\to&\aug{F^\times}\otimes_{\gr{\Z}{F^\times}}I^n,\\
 \pfist{a_0,\ldots,a_{n+1}}&\mapsto & \ffist{a_0}\otimes\pfist{a_1,\ldots,a_{n+1}}.
\end{eqnarray*}
The essential point here is that relations (a),(c) and (d) are clearly respected, while  in $\aug{F^\times}\otimes I^n$ we have
\begin{eqnarray*}
\ffist{a_0}\otimes\pfist{a_1,\ldots,a_{n+1}}& = & \ffist{a_0}\otimes \ffist{a_1}\cdot\pfist{a_2,\ldots,a_{n+1}}\\
  & = & \ffist{a_0}\ffist{a_1}\otimes\pfist{a_2,\ldots,a_{n+1}}\\
  & = & \ffist{a_1}\otimes \ffist{a_0}\cdot\pfist{a_2,\ldots,a_{n+1}}\\
  & = & \ffist{a_1}\otimes \pfist{a_0,a_2,\ldots,a_{n+1}}
\end{eqnarray*}
so that the relation (b) is also satisfied. 

Since the diagram
\begin{eqnarray*}
\xymatrix{
\aug{F^\times}\otimes_{\gr{\Z}{F^\times}} I^{n+1}\ar[rr]\ar[dr]
&&
I^{n+2}\ar[dl]_-{\phi_n}\\
&\aug{F^\times}\otimes_{\gr{\Z}{F^\times}}I^n&\\
}
\end{eqnarray*}
commutes, statement $B(n)$ now follows, and the theorem is proven.
 \end{proof}
\begin{lem}\label{lem:comm}
The square 
\begin{eqnarray*}
\xymatrix{
\filt{1}{\hoz{3}{\genl{3}{F}}}\ar[r]\ar[d]^{\ee}
&\ho{1}{F^\times}{\hoz{2}{\specl{3}{F}}}\ar[d]^{\mu}\\
\milk{3}{F}\ar[r]&\milkt{3}{F}\\
}
\end{eqnarray*}
commutes.
\end{lem}
\begin{proof}
We set $ST_n:=\ker{\det: T_n\to F^\times}$ so that $T_n$ decomposes as $T_1\times ST_n$, where $T_1=\set{(a,1,\ldots,1)}{a\in F^\times}
\subset T_n$.

The inclusion $ST_3\to \specl{3}{F}$ induces a map 
\[
\ext{2}{ST_3}=\hoz{2}{ST_3}\to \hoz{2}{\specl{3}{F}}\cong \milk{2}{F}
\]
 which, by 
\ref{thm:suslin} (3), is given by
\begin{eqnarray}\label{form:sl3}
x\wedge y \mapsto
-\{ x_2,y_3\} - \{ x_3,y_2\} -2\cdot\left( \{ x_2,y_2\} + \{ x_3,y_3\} \right)
\end{eqnarray}
where $x=(x_1,x_2,x_3), \ y=(y_1,y_2,y_3)$ with $\prod x_i= \prod y_i=1$.

The Hochschild-Serre spectral sequence of the extension
\[
1\to ST_3\to T_3\to F^\times \to 1
\]
maps to the corresponding spectral sequence for $\genl{3}{F}$. The induced filtration on $\hoz{\bullet}{T_3}$ 
is compatible with the K\"unneth 
decomposition of the integral homology of $T_3=T_1\times ST_2$ 
and the related decomposition of $\ext{\bullet}{T_3}$. In particular we have a map of short exact sequences
\begin{eqnarray*}
\xymatrix{
0\ar[r]&\ext{3}{ST_3}\ar[r]\ar[d]&\filt{1}{\ext{3}{T_3}}\ar[r]\ar[d]&T_1\otimes\ext{2}{ST_3}\ar[r]\ar[d]&0\\
0\ar[r]&\barhoz{3}{\specl{3}{F}}\ar[r]&\filt{1}{\hoz{3}{\genl{3}{F}}}\ar[r]& T_1\otimes\hoz{2}{\specl{3}{F}}\ar[r]&0
}
\end{eqnarray*}

The upper sequence is split by the natural inclusion $T_1\otimes \ext{2}{ST_2} \to \ext{3}{T_3}$. 

Using \ref{thm:suslin} (3)
again, the composite 
\[
T_1\otimes \ext{2}{ST_3}\to \filt{1}{\ext{3}{T_3}}\to\filt{1}{\hoz{3}{\genl{3}{F}}}\to\milk{3}{F}
\]
is given by
\[
(a,1,1)\otimes\left((x_1,x_2,x_3)\wedge (y_1,y_2,y_3)\right)\mapsto \{ a,x_2,y_3\} + \{ a, x_3,y_2\}.
\]
In particular, this composite map is surjective.

On the other hand, formula (\ref{form:sl3}) shows that the composite
\[
T_1\otimes \ext{2}{ST_3}\to T_1\otimes\hoz{2}{\specl{3}{F}}\to \milk{3}{F}
\]
sends $(a,1,1)\otimes\left((x_1,x_2,x_3)\wedge (y_1,y_2,y_3)\right)$ to the element 
\[
-\{ a,x_2,y_3\} - \{ a, x_3,y_2\}-2\cdot\left( \{a, x_2,y_2\}  + \{a, x_3,y_3\}\right).
\]

In view of diagram (\ref{diag:23}), this proves the lemma.
\end{proof}

\begin{cor}\label{cor:main}
For any infinite field $F$, the natural inclusion $\specl{2}{F}\to\specl{3}{F}$ induces an exact sequence 
\begin{eqnarray*}
\xymatrix{
\barhoz{3}{\specl{2}{F}}\ar[r]&\barhoz{3}{\specl{3}{F}}\ar[r]^-{\ee}&2\cdot\milk{3}{F}\ar[r]&0\\
}
\end{eqnarray*}

\end{cor}
\begin{proof}
This follows from the snake lemma applied to Diagram (\ref{diag:23}), Theorem \ref{thm:main} and Lemma \ref{lem:comm}. 

\end{proof}

\section{Homology Stability for $\hoz{3}{\specl{n}{F}}$}\label{sec:stab}

By Lemma \ref{lem:d2},  $\barhoz{3}{\specl{3}{F}}=\ho{0}{F^\times}{\hoz{3}{\specl{3}{F}}}$. In this 
section we prove that the natural action of $F^\times$ on $\hoz{3}{\specl{3}{F}}$ is trivial 
(and hence that $\barhoz{3}{\specl{3}{F}}=\hoz{3}{\specl{3}{F}}$). 

Let $\Xgen{q} = \xgen{q}{F}$ be the set of all $(q+1)$-tuples $(x_0,\ldots,x_q)$ of points of $\proj{2}{F}$ in general position 
(in the sense that no three of the $x_i$ are collinear in $\proj{2}{F}$). 

Thus $\Xgen{q}$ is naturally a $\genl{3}{F}$-set. We let $\Cgen{q}=\cgen{q}{F}:=\Z[\Xgen{q}]$ be the corresponding permutation module. 
Let $d:\Cgen{q}\to\Cgen{q-1}$ be the standard simplicial boundary map 
\[
(x_0,\ldots,x_q)\mapsto \sum_{i=0}^q(-1)^i(x_0,\ldots,\hat{x_i},\ldots,x_q).
\] 

The natural augmentation $\epsilon: \Cgen{\bullet}\to \Z$ gives a resolution of $\Z$ by $\genl{3}{F}$-modules; i.e., $\epsilon$ is a weak equivalence of complexes of 
$\GR{\Z}{\genl{3}{F}}$-modules. (This is standard: see, for example, section 1 of 
\cite{hutchinson:mat}.)
 
Restricting the action to 
$\specl{3}{F}$, there is a resulting spectral sequence of the form 
\[
E^1_{p,q}=\ho{p}{\specl{3}{F}}{\Cgen{q}}\imp\hoz{p+q}{\specl{3}{F}}
\]
with differentials $d^r:E^r_{p,q}\to E^r_{p+r-1,q-r}$. The differential $d^1:E^1_{p,q}-E^1_{p,q-1}$ 
is the map induced on homology by the boundary map $d$.

Since the $\Cgen{q}$ are permutation modules, the terms $E^1_{p,q}$ and the differentials $d^1$ can be explicitly computed:

We recall the following basic principles (see, for example, \cite{hutchinson:mat}): 
If $G$ is a group and if $X$ is a $G$-set, then Shapiro's Lemma says 
that
\[
\ho{p}{G}{\Z[X]}\cong \bigoplus_{y\in X/G}\ho{p}{G_y}{\Z},
\]
the isomorphism being induced by the maps 
\[
\ho{p}{G_y}{\Z}\to \ho{p}{G}{\Z[X]}
\]
described at the level of chains by
\[
B_p\otimes_{\gr{\Z}{G_y}}\Z\to B_p\otimes_{\gr{\Z}{G}}\Z[X],\quad z\otimes 1\mapsto z\otimes y.
\]

Let $X_i$, $i=1,2$ be transitive $G$-sets. Let $x_i\in X_i$ and let $H_i$ be the stabiliser of $x_i$, $i=1,2$. Let $\phi:\Z[X_1]\to\Z[X_2]$ 
be a map of $\Z[G]$-modules with
\[
\phi(x_1)=\sum_{g\in G/H_2}n_ggx_2,\qquad \mbox{ with } n_g\in\Z.
\] 
Then the induced map $\phi_\bullet:\ho{\bullet}{H_1}{\Z}\to\ho{\bullet}{H_2}{\Z}$ is given by the formula
\begin{eqnarray}\label{formula}
\phi_\bullet(z)
=\sum_{g\in H_1\backslash G/H_2}n_g\mathrm{cor}^{H_2}_{g^{-1}H_1g\cap H_2}
\mathrm{res}^{g^{-1}H_1g}_{g^{-1}H_1g\cap H_2}\left(g^{-1}\cdot z\right)
\end{eqnarray}

There is an obvious extension of this formula to non-transitive $G$-sets.

\subsection{The case $q=0$:}

We have $\Xgen{0}=\proj{2}{F}$. $\specl{3}{F}$ acts transitively on $\Xgen{0}$ and the stabiliser of $\p{e_1}$   is the group
\[
\sgen{0}:
=\left\{
\left(
\begin{array}{ccc}
\det(A)^{-1}&*&*\\
0&a&b\\
0&c&d\\
\end{array}
\right)
\left|
\quad
A= 
\left(
\begin{array}{cc}
a&b\\
c&d\\
\end{array}
\right)
\in \genl{2}{F}
\right.
\right\}.
\]

By Lemma \ref{lem:tb} above, the natural inclusion 
\[
\genl{2}{F}\to \sgen{0}, \quad A\mapsto
\left(
\begin{array}{cc}
\det{A}^{-1}&0\\
0&A\\
\end{array}
\right)
\]
induces a homology isomorphism $\hoz{p}{\genl{2}{F}}\cong\hoz{p}{\sgen{0}}$, and thus
\[
E^1_{p,0}=\ho{p}{\specl{3}{F}}{\Cgen{0}}\cong\hoz{p}{\sgen{0}}\cong\hoz{p}{\genl{2}{F}}.
\]

The edge homomorphisms 
\[
\hoz{p}{\genl{2}{F}}=E^1_{p,0}\to E^\infty_{p,0}\to \hoz{p}{\specl{3}{F}}
\]
are the maps on integral homology induced by the natural inclusion
\[
\genl{2}{F}\to \specl{3}{F}, \quad A\mapsto
\left(
\begin{array}{cc}
\det{A}^{-1}&0\\
0&A\\
\end{array}
\right)
\]
and so $E^\infty_{p,0}$ is isomorphic to the image of $\hoz{p}{\genl{2}{F}}$ in $\hoz{p}{\specl{3}{F}}$.

\subsection{The case $q=1$:}

$\Xgen{1}$ consists of pairs $(x_0,x_1)$ of distinct points of $\proj{2}{F}$. The action of $\specl{3}{F}$ is again transitive and the 
stabiliser of $(\p{e_1},\p{e_2})$ is the group

\[
\sgen{1}:
=\left\{
\left(
\begin{array}{ccc}
a&0&*\\
0&b&*\\
0&0&c\\
\end{array}
\right)
\left|
\quad
abc=1\mbox{ in } F^\times
\right.
\right\}.
\]

By Lemma \ref{lem:tb} again, the inclusion $ST_3\to \sgen{1}$ induces an isomorphism on integral homology and so
\[
E^1_{p,1}=\ho{p}{\specl{3}{F}}{\Cgen{1}}\cong\hoz{p}{\sgen{1}}\cong\hoz{p}{ST_3}\cong \hoz{p}{F^\times\times F^\times}.
\]

\subsection{The case $q=2$:}

$\Xgen{2}$ consists of `triangles' -- triples of non-collinear points -- in $\proj{2}{F}$. Again, the action of $\specl{3}{F}$ is 
transitive, and the stabiliser of $(\p{e_1},\p{e_2},\p{e_3})$ is the group of diagonal matrices, $ST_3$.  Thus

\[
E^1_{p,2}=\ho{p}{\specl{3}{F}}{\Cgen{2}}\cong \hoz{p}{ST_3}\cong \hoz{p}{F^\times\times F^\times}.
\]

\subsection{The case $q=3$:}

$\Xgen{3}$ consists of $4$-tuples, $(x_0,x_1,x_2,x_3)$ of points in $\proj{2}{F}$ in general position. 

\begin{lem}\label{lem:pi}
Let $\tilde{T_3}$ be the multiplicative group $\{ (a,b,c)\in F^3\ |\ abc\not=0 \}$.
Let $D_3\subset \tilde{T_3}$ be 
the subgroup generated by $\{ (a,b,c)\ |\ abc=1 \}$ and $\Delta(F^\times):=\{ (a,a,a)\ |\ a\in F^\times\}$.

Then the `determinant map'
 $\ddet:\tilde{T_3}\to F^\times/(F^\times)^3, (a,b,c)\mapsto abc\pmod{(F^\times)^3}$ induces an isomorphism
\(
\tilde{T_3}/D_3\cong F^\times/(F^\times)^3.
\)
\end{lem}

\begin{lem}\label{lem:x3}
Let $x=(x_0,x_1,x_2,x_3)\in \Xgen{3}$, and choose $v_i\in F^3\setminus \{ 0\}$ with  $x_i=\p{v_i}$. Let
\[
V:=(v_0|v_1|v_2)\in\genl{3}{F}.
\]

Then $V^{-1}v_3\in \tilde{T_3}$.  Let 
\[
\lambda(x):=\ddet(V^{-1}v_3)\cdot\det{V}\in F^\times/(F^\times)^3.
\]

Then $\lambda(x)$ is well-defined (i.e. independent of the choice of the $v_i$) and $x,y\in \Xgen{3}$ are in the same 
$\specl{3}{F}$-orbit if and only if $\lambda(x)=\lambda(y)$. 
\end{lem}
\begin{proof}
Observe first that $(\p{e_1},\p{e_2},\p{e_3},\p{v})\in \Xgen{3}$ if and only if $v\in \tilde{T_3}$.

Since the stabiliser of $(\p{e_1},\p{e_2},\p{e_3})$ is $ST_3\subset \specl{3}{F}$, it follows from Lemma \ref{lem:pi} that 
$(\p{e_1},\p{e_2},\p{e_3},\p{v})$ and $(\p{e_1},\p{e_2},\p{e_3},\p{w})$ are in the same $\specl{3}{F}$-orbit if and only if 
$\ddet(v)=\ddet(w)$.

Since $\specl{3}{F}$ acts transitively on $\Xgen{2}$, it follows that there exists $A\in\specl{3}{F}$ with 
\[
A\cdot (\p{e_1},\p{e_2},\p{e_3},\p{w})=x= (\p{v_0},\p{v_1},\p{v_2},\p{v_3}).
\]

Thus there exists $c_1,c_2,c_3\in F^\times$ with $Ae_i=c_iv_{i-1}$ for $i=1,2,3$

It follows that 
\[
A=[c_1v_0|c_2v_1|c_3v_2]=
V\cdot
\begin{pmatrix}
c_1&0&0\\
0&c_2&0\\
0&0&c_3\\
\end{pmatrix}
\]. 

Furthermore, $c_1c_2c_3=1/\det(V)$ since $\det(A)=1$. 

We have $Aw=cv_3$ for some $c\in F^\times$. Thus
\[
V\cdot 
\begin{pmatrix}
c_1&0&0\\
0&c_2&0\\
0&0&c_3\\
\end{pmatrix}
w = cv_3\qquad \imp \begin{pmatrix}
c_1&0&0\\
0&c_2&0\\
0&0&c_3\\
\end{pmatrix}w=c(V^{-1}v_3)\in \tilde{T_3}.
\]

Applying $\ddet$ to this now gives 
\[
(c_1c_2c_3)\ddet(w)=c^3\ddet(V^{-1}v_3)\qquad \imp \ddet(w)=\ddet(v^{-1}v_3)\det(V)\in F^\times/(F^\times)^3.
\]
\end{proof}

It follows that the $\specl{3}{F}$-orbits of $\Xgen{3}$ are indexed by $F^\times/(F^\times)^3$, where $\bar{a}\in F^\times/(F^\times)^3$
corresponds to the orbit of the element
\[
[\bar{a}]:=(\p{e_1},\p{e_2},\p{e_3},[1,1,a]).
\]

The stabiliser of $[\bar{a}]$ is $\Delta(F^\times)\cap \specl{3}{F}=\mu_3(F)$.

Thus 
\[
E^1_{p,3}=\ho{p}{\specl{3}{F}}{\Xgen{3}}\cong \bigoplus_{\bar{a}\in F^\times/(F^\times)^3}\hoz{p}{\mu_3(F)}\cdot [\bar{a}].
\]

In particular, if $\zeta_3\not\in F$ then $E^1_{p,3}=0$ for $p>0$ and, in any case,
\[
E^1_{0,3}\cong \bigoplus_{\bar{a}\in F^\times/(F^\times)^3}\Z\cdot [\bar{a}] = \Z[F^\times/(F^\times)^3].
\]

\subsection{The case $q=4$:}

Given $a,b,c,d\in F^\times$ satisfying $(b-a)(d-c)(ad-bc)\not=0$, then
\[
\begin{Vmatrix}
a&b\\
c&d\\
\end{Vmatrix}
:=
(\p{e_1},\p{e_2},\p{e_3},[1,a,c],[1,b,d])\in \Xgen{4}.
\]

Let $Y_4$ be the set of all such $5$-tuples. Every element of $\Xgen{4}$ belongs to the $\specl{3}{F}$-orbit of an element of this type. 
Furthermore, two elements
\[
\begin{Vmatrix}
a&b\\
c&d\\
\end{Vmatrix},\ 
\begin{Vmatrix}
a'&b'\\
c'&d'\\
\end{Vmatrix} \in Y_4
\] 
belong to the same $\specl{3}{F}$-orbit if and only if there exist $r,s\in F^\times$ satisfying
\begin{eqnarray*}
rs\in (F^\times)^3,\qquad \frac{a'}{a}=\frac{b'}{b}=r,\qquad \frac{c'}{c}=\frac{d'}{d}=s.
\end{eqnarray*}
Thus, letting $Z_4$ equal $Y_4$ modulo this equivalence relation we have 
\[
E^1_{p,4}=\bigoplus_{{\tiny \begin{Vmatrix}
a&b\\
c&d\\
\end{Vmatrix}}\in Z_4}\hoz{p}{\mu_3(F)}\cdot {\scriptsize \begin{Vmatrix}
a&b\\
c&d\\
\end{Vmatrix}}.
\]

Thus the $E^1$-term of the spectral sequence has the form

 \begin{eqnarray*}
 \xymatrix{
 \bigoplus_{Z_4}\Z\cdot {\scriptsize \begin{Vmatrix}
a&b\\
c&d\\
\end{Vmatrix}}
 &\cdots &\cdots &\cdots \\
 \Z[F^\times/(F^\times)^3]
 &
 \cdots
 &\cdots &\cdots \\
 \Z&ST_3&\hoz{2}{ST_3}&\cdots\\
 \Z
 &ST_3
 &\hoz{2}{ST_3}
 &\hoz{3}{ST_3}\\
\Z&
F^\times
&
\hoz{2}{\genl{2}{F}}
&\hoz{3}{\genl{2}{F}}
 }
 \end{eqnarray*}

\subsection{Some $E^\infty$-terms}
To compute $\hoz{3}{\specl{3}{F}}$, we need to calculate the terms $E^\infty_{p,q}$ with $p+q=3$. 

As already observed $E^\infty_{3,0}$ is the image of $\hoz{3}{\genl{2}{F}}$. We will show that the other $E^\infty$-terms on the line 
$p+q=3$ are $0$.

We have $E^\infty_{0,3}=E^2_{0,3}= \coker{d^1}:E^1_{0,4}\to E^1_{0,3}$.

\begin{lem}
The map 
\[
d^1:\bigoplus \Z\cdot {\scriptsize \begin{Vmatrix}
a&b\\
c&d\\
\end{Vmatrix}} = E^1_{0,4}\to E^1_{0,3}=\bigoplus_{a\in F^\times/(F^\times)^3}\Z\cdot [a]
\]
is surjective.
\end{lem} 
\begin{proof}
\begin{eqnarray*}
d^1\left({\scriptsize \begin{Vmatrix}
a&b\\
c&d\\
\end{Vmatrix}}\right)&=&d^1(\p{e_1},\p{e_2},\p{e_3},[1,a,c],[1,b,d])\\
&=& \lambda(\p{e_2},\p{e_3},[1,a,c],[1,b,d])-\lambda(\p{e_1},\p{e_3},[1,a,c],[1,b,d])\\
&&+\lambda(\p{e_1},\p{e_2},[1,a,c],[1,b,d])-\lambda(\p{e_1},\p{e_2},\p{e_3},[1,b,d])+\lambda(\p{e_1},\p{e_2},\p{e_3},[1,a,c])\\
&=& [(b-a)(d-c)]-[ab(a-b)(ad-bc)]+[(c-d)(ad-bc)cd]-[bd]+[ac].
\end{eqnarray*}
It follows, in particular, that if $b^2\not= a^2$, then
\[
d^1\left({\scriptsize \begin{Vmatrix}
a&b\\
b&a\\
\end{Vmatrix}}\right)
=
[(b-a)^2]=[(b-a)^{-1}]
\]
so that $d^1$ is surjective as claimed (and $E^2_{0,3}=0$).
\end{proof}

\begin{lem}\label{lem:triv}
The spectral sequence admits an $F^\times$-action compatible with the action of $F^\times$ on $\hoz{n}{\specl{3}{F}}$. The induced 
$F^\times$-action on the terms $E^r_{p,1}$ and $E^r_{p,2}$ is trivial for all $p$ and all $r\geq 1$. 
\end{lem}

\begin{proof} 
Let $\delta:F^\times \to \genl{3}{F}$ be the map $a\mapsto \mathrm{diag}(1,1,a)$ (a splitting of the determinant map). 

Recall that the spectral sequence $E^r_{\bullet,\bullet}(3)$ is the second spectral sequence associated to the double complex
\[
E^0_{p,q}(3):= B_p(\specl{3}{F})\otimes_{\GR{\Z}{\specl{3}{F}}}C_q.
\]   

Let $F^\times$ act on $E^0_{p,q}$ by
\[
a\cdot\left([g_1|\cdots|g_p]\otimes (x_0,\ldots,x_q)\right)=[\delta(a)g_1\delta(a)^{-1}|\cdots |\delta(a)g_p\delta(a)^{-1}]
\otimes \delta(a)\cdot (x_0,\ldots,x_q).
\]

This action makes $E^0_{\bullet,\bullet}$ into a double complex of $\gr{\Z}{F^\times}$-modules and the induced actions on 
$E^1_{p,q}=\ho{p}{\specl{3}{F}}{C_q}$ are the natural actions derived from the $\genl{3}{F}$-action on $C_q$ and the short exact sequence
$1\to \specl{3}{F}\to\genl{3}{F}\to F^\times\to 1$.

At the level of chains, the isomorphisms  $\hoz{p}{ST_3}\cong\hoz{p}{S_q}\cong \ho{p}{\specl{3}{F}}{C_q}$ for $q=1,2$ are given by
\begin{eqnarray*}
B_p(ST_3)\otimes_{\GR{\Z}{ST_3}}\Z
&\to
B_p(S_q)\otimes_{\GR{\Z}{S_q}}\Z
\to
&B_p(\specl{3}{F})\otimes_{\GR{\Z}{\specl{3}{F}}}C_q\\
\ [g_1|\cdots |g_p]\otimes 1 
&
\mapsto
&
[g_1|\cdots |g_p]\otimes e(q)\\
\end{eqnarray*}
where $ e(1):= \p{e_1}$ and $e(2):=(\p{e_1},\p{e_2})$.

If the domain here is endowed with the trivial $\gr{\Z}{F^\times}$-module structure, then this is a module homomorphism, since 
$\delta(F^\times)$ centralizes $ST_3$ and stabilizes the $e(q)$, $q=1,2$.
\end{proof}
\begin{cor}\label{cor:triv}
There is an  exact sequence of $\gr{\Z}{F^\times}$-modules 
\[
\hoz{3}{\genl{2}{F}}\to\hoz{3}{\specl{3}{F}}\to Q\to 0
\]
where $\aug{F^\times}^2Q=0$.
\end{cor}
\begin{proof}
Since $E^\infty_{0,3}=0$ and $E^\infty_{3,0}$ is the image of $\hoz{3}{\genl{2}{F}}$, the spectral sequence yields an exact sequence of 
of the type described where $Q$ fits into a short exact sequence (of $\gr{\Z}{F^\times}$-modules)
\[
0\to E^\infty_{2,1}\to Q\to E^\infty_{1,2}\to 0.
\]
By Lemma \ref{lem:triv}, $\aug{F^\times} E^\infty_{2,1}=\aug{F^\times} E^\infty_{1,2}=0$ and thus $\aug{F^\times}^2Q=0$.
\end{proof}
\begin{lem}\label{lem:h3}
For all infinite fields $F$, $\ho{0}{F^\times}{\hoz{3}{\specl{3}{F}}}=\hoz{3}{\specl{3}{F}}$.
\end{lem}
\begin{proof}
This amounts to showing that the natural action of $F^\times$ on $\hoz{3}{\specl{3}{F}}$ is trivial. 

Since the action of $F^\times$ on $\hoz{3}{\genl{2}{F}}$ is trivial, it is enough to prove that $Q=0$.

Now
\[
\barhoz{3}{\specl{3}{F}}\cong \ho{0}{F^\times}{\hoz{3}{\specl{3}{F}}}=\frac{\hoz{3}{\specl{3}{F}}}{\aug{F^\times}\cdot \hoz{3}{\specl{3}{F}}}.
\]

Observe that the composite $\ext{3}{ST_3}\to \hoz{3}{\genl{2}{F}}\to\barhoz{3}{\specl{3}{F}}\to 2\milk{3}{F}$ is surjective, since for 
example the element $(a,1,a^{-1})\wedge (b,b^{-1},1)\wedge (c,c^{-1},1)$ maps to $\{ a^{-1},b^{-1},c\}+\{ a^{-1},b,c^{-1}\}=-2\{ a,b,c\}$.
Since the image of $\beta: \hoz{3}{\genl{2}{F}}\to\barhoz{3}{\specl{3}{F}}$ also contains the image of $\hoz{3}{\specl{2}{F}}\to 
\barhoz{3}{\specl{3}{F}}$, $\beta$ is surjective by Theorem \ref{thm:main}. 

In the diagram
\begin{eqnarray*}
\xymatrix{
&0\ar[d]&&\\
&\aug{F^\times}\cdot\hoz{3}{\specl{3}{F}}\ar[d]\ar[dr]^{\alpha}&&\\
\hoz{3}{\genl{2}{F}}\ar[r]\ar[dr]^{\beta}&\hoz{3}{\specl{3}{F}}\ar[r]\ar[d]&Q\ar[r]&0\\
&\barhoz{3}{\specl{3}{F}}\ar[d]&&\\
&0&&\\
}
\end{eqnarray*}
$\beta$ is surjective and hence so is $\alpha$. It follows that $Q=\aug{F^\times} Q$. Thus $Q=\aug{F^\times}^2 Q=0$ 
using Corollary \ref{cor:triv}.
\end{proof}

\begin{thm}
If $F$ is an infinite field then
\[
\hoz{3}{\specl{3}{F}}\cong\hoz{3}{\specl{4}{F}}\cong\cdots \cong \hoz{3}{\specl{}{F}}
\]
and there is a natural exact sequence
\[
\hoz{3}{\specl{2}{F}}\to\hoz{3}{\specl{3}{F}}\to 2\milk{3}{F}\to 0.
\]
\end{thm}
\begin{proof}
In view of Lemma \ref{lem:d2}, Corollary \ref{cor:main} and Lemma \ref{lem:h3} it remains to show that the action of $F^\times$ on
$\hoz{3}{\specl{n}{F}}$ is trivial for all $n\geq 4$. This can be accomplished in the same way as the case $n=3$, but the 
computations are simpler: Consider the spectral sequence associated to the action of $\specl{n}{F}$ on the complex $C_\bullet(n)$ where 
$C_q(n)$ has as a basis the set of $(q+1)$-tuples of points of $\proj{n-1}{F}$ in general position. The $E^\infty$ 
terms on the base are the 
image of the homology of $\genl{n-1}{F}$ as before. It is now 
easy to show by direct computation, using formula (\ref{formula}) above, that the terms $E^2_{0,3}$, $E^2_{1,2}$ and 
$E^2_{2,1}$ are all zero (this is not the case when $n=3$). Alternatively one can show as above that $F^\times$ acts trivially on 
$E^1_{1,2}$ and $E^1_{2,1}$ and argue as in the case $n=3$. 
\end{proof}

\section{Concluding Remarks: Indecomposable $K_3$.}

We let $\kf{\bullet}{F}$ be the Quillen $K$-theory of the field $F$. There is a natural homomorphism of graded rings 
$\milk{\bullet}{F}\to \kf{\bullet}{F}$, which is an isomorphism in dimensions at most $2$. Let $\qmilk{\bullet}{F}$ denote the image of this
homomorphism.The kernel of the surjective homomorphism $\milk{3}{F}\to\qmilk{3}{F}$ is known to be $2$-torsion, and it is an open question 
whether this homomorphism is an isomorphism. 

The indecomposable $K_3$ of the field $F$ is the group $\indk{3}{F}:=\kf{3}{F}/\qmilk{3}{F}$. 

Suslin has shown that the Hurewicz homomorphism 
of algebraic $K$-theory yields a surjective homomorphism $\kf{3}{F}\to\hoz{3}{\spcl{F}}$ whose
kernel is $\{ -1\}\cdot\kf{2}{F}\subset\qmilk{3}{F}$ (\cite{sus:bloch}, Corollary 5.2 or see 
Knudson, \cite{knudson:book}). 
It follows that there is an induced surjective map 
$\hoz{3}{\specl{3}{F}}=\hoz{3}{\spcl{F}}\to \indk{3}{F}$. In view of the results above, we have 

\begin{lem} For any infinite field $F$, the natural homomorphism  $\hoz{3}{\specl{2}{F}}\to \indk{3}{F}$
is surjective.
\end{lem}
\begin{proof}
We have the following diagram (defining the maps $\gamma$ and $\delta$) with exact row and column
\begin{eqnarray*}
\xymatrix{
&\milk{3}{F}\ar[d]\ar[dr]^{\gamma}&&\\
\hoz{3}{\specl{2}{F}}\ar[r]\ar[dr]^{\delta}&\kf{3}{F}/(\{ -1\}\cdot\kf{2}{F})\ar[r]\ar[d]&2\milk{3}{F}\ar[r]&0\\
&\indk{3}{F}\ar[d]&&\\
&0&&\\
}
\end{eqnarray*}
Now Suslin has shown that the composite 
\begin{eqnarray*}
\xymatrix{
\milk{3}{F}\ar[r]
&
\kf{3}{F}\ar[r]^-{H}
&
\hoz{3}{\gnl{F}}\ar[r]
&\milk{3}{F}
}
\end{eqnarray*}
(where $H$ is the Hurewicz homomorphism) is just multplication by $2$ (\cite{sus:homgln}, Corollary 4.4).
 Since, by the arguments above, the diagram
\begin{eqnarray*}
\xymatrix{
\kf{3}{F}\ar[r]\ar[rd]^-{H}
&
\hoz{3}{\spcl{F}}\ar[r]\ar[d]
&
2\milk{3}{F}\ar[d]\\
&\hoz{3}{\gnl{F}}\ar[r]
&
\milk{3}{F}\\
}
\end{eqnarray*}
commutes, it follows that $\gamma$ is multiplication by $2$ and so is surjective. This implies the surjectivity of $\delta$.
\end{proof}

 Regarding the kernel of the map $\hoz{3}{\specl{2}{F}}\to\indk{3}{F}$, Suslin has asked:
\begin{quote}
Is it the case that $\ho{0}{F^\times}{\hoz{3}{\specl{2}{F}}}$ coincides with 
$\indk{3}{F}$ for all (infinite) fields $F$? (see for example Sah, \cite{sah:discrete3}). 
\end{quote}

In \cite{sah:discrete3}, Sah proves that
the answer is affirmative when $F$ satisfies $(F^\times)^6=F^\times$, and also for $F=\R$. More generally, he states that the natural map 
 $\ho{0}{F^\times}{\hoz{3}{\specl{2}{F}}}\to \indk{3}{F}$ is an 
 isomorphism modulo the Serre class of torsion abelian groups annihilated 
by a power of $6$.

Since then, P. Elbaz-Vincent, \cite{pev:indk3},
 proved that for any semi-local ring $R$ with infinite residue fields there is an isomorphism
\[
\indk{3}{R}\otimes\Q\cong\ho{0}{R^\times}{\ho{3}{\specl{2}{R}}{\Q}}.
\]  
More recently, B. Mirzaii, \cite{mirzaii:3rd}, has improved on Sah's result by showing that if $F$ is an infinite field, then 
\[
\indk{3}{F}\otimes\Z[1/2]\cong\ho{0}{F^\times}{\ho{3}{\specl{2}{F}}{\Z[1/2]}}
\] 
and that if $F^\times$ is $2$-divisible, then 
\[
\indk{3}{F}\cong \hoz{3}{\specl{2}{F}}.
\]

In view of the results above, we note that an affirmative answer to Suslin's question for a given infinite field $F$ is equivalent 
to the truth of the following three statements:
\begin{enumerate}
\item[(i)] The image of $\hoz{3}{\specl{2}{F}}$ in $\hoz{3}{\specl{3}{F}}$ is isomorphic to $\indk{3}{F}$.
\item[(ii)] The map $\hoz{3}{\genl{2}{F}}\to\hoz{3}{\genl{3}{F}}$ is injective, and
\item[(iii)] The differential $d^2_{2,2}(2)$ is zero.
\end{enumerate}

\section{Acknowledgements}
The work in this article was partially funded by  the Science Foundation Ireland Research Frontiers Programme grant 05/RFP/MAT0022.
\bibliographystyle{alpha}
\bibliography{HomologySL}

\begin{thebibliography}{Mor04b}

\bibitem[BM99]{barge:morel}
Jean Barge and Fabien Morel.
\newblock Cohomologie des groupes lin\'eaires, {$K$}-th\'eorie de {M}ilnor et
  groupes de {W}itt.
\newblock {\em C. R. Acad. Sci. Paris S\'er. I Math.}, 328(3):191--196, 1999.

\bibitem[Bro82]{brown:coh}
Kenneth~S. Brown.
\newblock {\em Cohomology of groups}, volume~87 of {\em Graduate Texts in
  Mathematics}.
\newblock Springer-Verlag, New York, 1982.

\bibitem[Cat04]{cath:milnor}
Jean-Louis Cathelineau.
\newblock Projective configurations, homology of orthogonal groups, and
  {M}ilnor {$K$}-theory.
\newblock {\em Duke Math. J.}, 121(2):343--387, 2004.

\bibitem[EV98]{pev:indk3}
Philippe Elbaz-Vincent.
\newblock The indecomposable {$K\sb 3$} of rings and homology of {${\rm SL}\sb
  2$}.
\newblock {\em J. Pure Appl. Algebra}, 132(1):27--71, 1998.

\bibitem[HT08]{hutchinson:tao}
Kevin Hutchinson and Liqun Tao.
\newblock A note on {M}ilnor-{W}itt {$K$}-theory and a theorem of {S}uslin.
\newblock {\em Comm. Alg}, 36:2710--2718, 2008.

\bibitem[Hut90]{hutchinson:mat}
Kevin Hutchinson.
\newblock A new approach to {M}atsumoto's theorem.
\newblock {\em $K$-Theory}, 4(2):181--200, 1990.

\bibitem[Knu01]{knudson:book}
Kevin~P. Knudson.
\newblock {\em Homology of linear groups}, volume 193 of {\em Progress in
  Mathematics}.
\newblock Birkh\"auser Verlag, Basel, 2001.

\bibitem[Lam05]{lam:intro}
T.~Y. Lam.
\newblock {\em Introduction to quadratic forms over fields}, volume~67 of {\em
  Graduate Studies in Mathematics}.
\newblock American Mathematical Society, Providence, RI, 2005.

\bibitem[Maz05]{mazz:sus}
A.~Mazzoleni.
\newblock A new proof of a theorem of {S}uslin.
\newblock {\em $K$-Theory}, 35(3-4):199--211 (2006), 2005.

\bibitem[Mil70]{milnor:intro}
John Milnor.
\newblock Algebraic {$K$}-theory and quadratic forms.
\newblock {\em Invent. Math.}, 9:318--344, 1969/1970.

\bibitem[Mil71]{milnor:quad}
John Milnor.
\newblock {\em Introduction to algebraic {$K$}-theory}.
\newblock Princeton University Press, Princeton, N.J., 1971.
\newblock Annals of Mathematics Studies, No. 72.

\bibitem[Mir07]{mirzaii:3rd}
B.~Mirzaii.
\newblock Third homology of general linear groups.
\newblock (To appear, J. Algebra), Oct. 2007.

\bibitem[Mor04a]{morel:trieste}
Fabien Morel.
\newblock An introduction to {$\Bbb A\sp 1$}-homotopy theory.
\newblock In {\em Contemporary developments in algebraic $K$-theory}, ICTP
  Lect. Notes, XV, pages 357--441 (electronic). Abdus Salam Int. Cent. Theoret.
  Phys., Trieste, 2004.

\bibitem[Mor04b]{morel:puiss}
Fabien Morel.
\newblock Sur les puissances de l'id\'eal fondamental de l'anneau de {W}itt.
\newblock {\em Comment. Math. Helv.}, 79(4):689--703, 2004.

\bibitem[Mor06]{morel:a1}
Fabien Morel.
\newblock {$\Bbb A\sp 1$}-algebraic topology.
\newblock In {\em International Congress of Mathematicians. Vol. II}, pages
  1035--1059. Eur. Math. Soc., Z\"urich, 2006.

\bibitem[OVV07]{voevodsky:orlovvishik}
D.~Orlov, A.~Vishik, and V.~Voevodsky.
\newblock An exact sequence for {$K\sp M\sb \ast/2$} with applications to
  quadratic forms.
\newblock {\em Ann. of Math. (2)}, 165(1):1--13, 2007.

\bibitem[Sah89]{sah:discrete3}
Chih-Han Sah.
\newblock Homology of classical {L}ie groups made discrete. {III}.
\newblock {\em J. Pure Appl. Algebra}, 56(3):269--312, 1989.

\bibitem[Sus84]{sus:homgln}
A.~A. Suslin.
\newblock Homology of {${\rm GL}\sb{n}$}, characteristic classes and {M}ilnor
  {$K$}-theory.
\newblock In {\em Algebraic $K$-theory, number theory, geometry and analysis
  (Bielefeld, 1982)}, volume 1046 of {\em Lecture Notes in Math.}, pages
  357--375. Springer, Berlin, 1984.

\bibitem[Sus87]{sus:tors}
A.~A. Suslin.
\newblock Torsion in {$K\sb 2$} of fields.
\newblock {\em $K$-Theory}, 1(1):5--29, 1987.

\bibitem[Sus90]{sus:bloch}
A.~A. Suslin.
\newblock {$K\sb 3$} of a field, and the {B}loch group.
\newblock {\em Trudy Mat. Inst. Steklov.}, 183:180--199, 229, 1990.
\newblock Translated in Proc.\ Steklov Inst.\ Math.\ {\bf 1991}, no.\ 4,
  217--239, Galois theory, rings, algebraic groups and their applications
  (Russian).

\end{thebibliography}
\end{document}